\documentclass[11pt]{article}

\usepackage{amsmath,amssymb,amsthm,amsfonts,verbatim}
\usepackage{enumerate}
\usepackage{microtype}
\usepackage[all,2cell]{xy}
\usepackage{mathtools}
\CompileMatrices

\usepackage[top=1.3in,bottom=1.9in,left=1.4in,right=1.4in]{geometry}

\title{On the geometric nature of\\ characteristic classes of surface bundles}

\author{Thomas Church\footnotemark[1], Benson Farb\thanks{The first and second authors
    gratefully acknowledge support from the National Science
    Foundation.}  \ and Matthew Thibault}

\theoremstyle{plain}
\newtheorem{theorem}{Theorem}[section]
\newtheorem{question}[theorem]{Question}
\newtheorem{proposition}[theorem]{Proposition}

\newtheorem*{theorem:main}{Theorem \ref{theorem:MMMgeom}}
\newtheorem*{theorem:cobordism}{Theorem \ref{theorem:novikov}}
\newtheorem*{corollary:GT}{Corollary \ref{GTcorollary}}
\newtheorem*{corollary:obstruct}{Corollary \ref{corollary:obstruction}}
\newtheorem{corollary}[theorem]{Corollary}
\theoremstyle{definition}
\newtheorem{definition}[theorem]{Definition}

\newcommand{\nc}{\newcommand}
\nc{\dmo}{\DeclareMathOperator}

\nc{\Q}{\mathbb{Q}}
\nc{\R}{\mathbb{R}}
\nc{\Z}{\mathbb{Z}}
\nc{\C}{\mathbb{C}}
\nc{\cS}{\mathcal{S}}
\dmo{\Mod}{Mod}
\dmo{\Diff}{Diff}
\dmo{\Homeo}{Homeo}
\dmo{\dist}{dist}
\dmo\BDiff{BDiff}
\dmo\SO{SO}

\renewcommand{\epsilon}{\varepsilon}
\nc{\coloneq}{\mathrel{\mathop:}\mkern-1.2mu=}
\nc{\margin}[1]{\marginpar{\scriptsize #1}}
\nc{\para}[1]{\medskip\noindent\textbf{#1.}}

\begin{document}
\maketitle
\begin{abstract}
Each Morita--Mumford--Miller (MMM) class $e_n$ assigns to each genus $g\geq 2$ surface bundle $\Sigma_g\to  E^{2n+2}\to M^{2n}$ an integer $e_n^\#(E\to M)\coloneq \langle e_n,[M]\rangle\in\Z$.  We prove that when $n$ is odd 
the number $e_n^\#(E\to M)$ depends only on the diffeomorphism type of $E$, not on $g$, $M$ or the map $E\to M$.  More generally, we prove that $e_n^\#(E\to M)$ depends only on the cobordism class of $E$.  Recent work of Hatcher implies that this stronger statement is false when $n$ is even. If $E\to M$ is a holomorphic fibering of complex manifolds, we show that for every $n$ the number $e_n^\#(E\to M)$ only depends on the complex cobordism type of $E$. 

We give a general procedure to construct manifolds fibering as surface bundles in multiple ways, providing infinitely many examples to which our theorems apply.  As an application of our results we give a new proof of the rational case of a recent theorem of Giansiracusa--Tillmann \cite[Theorem A]{GT} that the odd MMM classes $e_{2i-1}$ vanish for any surface bundle which bounds a handlebody bundle.   We show how the MMM classes can be seen as obstructions to low-genus fiberings.  Finally, we discuss a number of open questions that arise from this work.
\end{abstract}

\section{Introduction}
The purpose of this paper is to explain our discovery of certain coincidences between certain 
characteristic numbers associated to surface bundles.  This can be considered a step towards understanding the geometric meaning of the Morita--Mumford--Miller classes.  

\para{Characteristic classes of surface bundles} 
Let $\Sigma_g$ be a closed oriented surface of genus $g\geq 2$.   A \emph{$\Sigma_g$--bundle} over a base space $B$ is a fiber bundle 
\begin{equation}
\label{eq:bundle1}
\Sigma_g\to E\to B
\end{equation}
with structure group $\Diff^+(\Sigma_g)$. When $E$ and $B$ are closed smooth manifolds, Ehresmann's theorem implies that any surjective submersion $\pi\colon E\to B$ whose fiber is $\Sigma_g$ (with a consistent orientation on $\ker d\pi$) is a $\Sigma_g$--bundle in this sense.

An $i$--dimensional \emph{characteristic class} $c$ for $\Sigma_g$--bundles is a natural  association of 
an element $c(E\to B)\in H^i(B;\Q)$ to every bundle as in \eqref{eq:bundle1}.  To be \emph{natural} means that if $\Sigma_g\to E'\to B'$ is any bundle and if $f\colon B\to B'$ is any continuous map covered by a oriented bundle map $E\to E'$ then \[c(E\to B)=f^\ast c(E'\to B').\]  If the association $c$ is defined for all $g\geq 2$ then $c$ is called a \emph{characteristic class of surface bundles}.  Characteristic classes are fundamental isomorphism invariants for surface bundles.   When $B$ is a closed, oriented, 
$i$--dimensional manifold, one obtains a numerical invariant $c^{\#}(E\to B)\in\Q$, called a \emph{characteristic number}, by pairing $c\in H^i(B;\Q)$ with the fundamental class $[B]\in H_i(B;\Q)$:
\[c^\#(E\to B)\coloneq \langle c(E\to B),[B]\rangle\in\Q.\]

\para{Notational convention}  Unless otherwise specified, all manifolds are assumed to be smooth and oriented, all bundles in this paper are assumed to be bundles of closed manifolds, and all maps are assumed to be orientation-preserving.   

\para{Morita--Mumford--Miller classes} The main examples of characteristic classes of $\Sigma_g$--bundles were given by Miller, Morita and Mumford (see, e.g.\ \cite{Mo}), and can be constructed as follows.   Let $\Sigma\to
E\overset{\pi}{\longrightarrow} B$ be any surface bundle. The kernel of
the differential $d\pi$ is a 2--dimensional oriented real vector
bundle $T\pi\to E$.  Associated to an oriented $2$--plane bundle is its \emph{Euler class} 
$e(T\pi)\in H^2(E)$;  we will denote it simply by $e\in H^2(E)$.  Let $\int_{\Sigma}\colon
H^*(E)\to H^{*-2}(B)$ be the Gysin map, given by integration along the
fiber $\Sigma$.  The \emph{$i$-th Morita--Mumford--Miller (or \emph{MMM}) class} 
$e_i(E\to B)\in H^{2i}(B)$ is defined by the formula
\[e_i(E\to B)=\int_\Sigma e^{i+1}.\]

The naturality of the Euler class implies that each $e_i$ is a characteristic class for $\Sigma_g$--bundles. Moreover, each $e_i$ is \emph{stable}; this means roughly that when it is possible to take the fiber-wise connected sum of $\Sigma_g\to E\to B$ with $\Sigma_1\to \Sigma_1\times B\to B$, the resulting bundle $\Sigma_{g+1}\to E'\to B$ satisfies $e_i(E'\to B)=e_i(E\to B)$. More strongly, each $e_i$ is \emph{primitive}, meaning roughly that when the fiber-wise connected sum of $\Sigma_g\to E_1\to B$ with $\Sigma_h\to E_2\to B$ exists, the resulting bundle $\Sigma_{g+h}\to E'\to B$ satisfies \[e_i(E'\to B)=e_i(E_1\to B)+e_i(E_2\to B).\] Generalizing a construction due to Kodaira and Atiyah, Morita
and Miller constructed explicit $\Sigma_g$--bundles to show that each 
polynomial $P(e_1,e_2,\ldots)$ yields a characteristic class which is nonzero for $g$ large enough (depending on $P$).

Madsen--Weiss \cite{MW} proved the remarkable theorem that the MMM classes yield all the stable characteristic classes:  any characteristic class of surface bundles which is stable must be a polynomial $P(e_1,e_2,\ldots)$ in the $e_i$, and any characteristic class of surface bundles which is primitive must be one of the MMM classes $e_i$ themselves.  In spite of this progress, many basic geometric questions concerning the MMM classes $e_i$ remain open.  

\para{Fibered 4--manifolds} The main result of this paper begins with an observation about $e_1$.  The Miller--Morita examples proving nontriviality of $e_1$ reduce to the original 
$4$--dimensional surface bundles $\Sigma_g\to E^4\to \Sigma_h$ constructed by Kodaira and Atiyah 
for certain specific $g,h$ (see e.g.\ \cite[\S4.3.3]{Mo}). It turns out that these
4--manifolds actually fiber as a surface bundle over a surface in multiple distinct
ways.  For example the Atiyah--Kodaira manifold $$\Sigma_4\to N^4\to \Sigma_{17}$$ also fibers 
as $$\Sigma_{49}\to N^4\to \Sigma_2.$$ 

The associated characteristic
classes 
\[e_1(N^4\to \Sigma_{17})\in H^2(\Sigma_{17};\Z)
\qquad\text{
  and }\qquad e_1(N^4\to \Sigma_2)\in H^2(\Sigma_2;\Z)\] are \emph{a
  priori} unrelated.  (Indeed, they come from cohomology classes living in two
different groups, $H^2(\BDiff^+(\Sigma_4);\Z)$ and
$H^2(\BDiff^+(\Sigma_{49});\Z)$ respectively; see Section~\ref{section:bordism}.)   The base spaces of the two fiberings of $N^4$ are also different manifolds.  However, pairing each of these 
characteristic classes with the fundamental class of the base space of the bundle gives 
(via a computation, e.g. following Hirzebruch \cite{Hi}):

\[\langle e_1(N^4\to \Sigma_{17}),[\Sigma_{17}]\rangle=96
\qquad\qquad \langle e_1(N^4\to \Sigma_2),[\Sigma_2]\rangle=96\]
and so the characteristic numbers of the two different bundles coincide.    

This coincidence is explained by the Hirzebruch Signature Formula, which implies 
for any fibering $\Sigma_g\to
N^4\to \Sigma_h$ the formula 
\begin{equation}
\label{eq:sigformula}
\langle
e_1(N^4\to \Sigma_h),[\Sigma_h]\rangle=3\cdot \sigma(N^4)
\end{equation}
where $\sigma(N^4)$ is the \emph{signature} of $N^4$, i.e.\ the signature of the intersection pairing on $H^2(N^4;\Q)$.   In particular, since the right hand side of \eqref{eq:sigformula} depends only on $N^4$ and not on the choice of fibering $N^4\to \Sigma_h$, the same holds for the left 
hand side.

\para{Geometric characteristic classes} With the above example in mind we
make the following definition.  

\begin{definition}[\textbf{Geometric characteristic class}]
\label{def}
An $i$--dimensional characteristic class $c$ for surface bundles is \emph{geometric} if the associated characteristic number $c^{\#}$ depends only on the diffeomorphism class of the total space of the bundle.  More precisely, whenever 
\[\Sigma_g\to E^{i+2}\to M^{i}\] and \[\Sigma_h\to X^{i+2}\to N^{i}\]
are surface bundles with diffeomorphic total spaces $E^{i+2}\approx X^{i+2}$, the associated characteristic numbers are equal:
\[c^{\#}(E^{i+2}\to M^i)=c^{\#}(X^{i+2}\to N^i)\in \Q.\]
\end{definition}

\bigskip
We also have the corresponding notions of \emph{geometric with respect to homeomorphisms, homotopy equivalences, cobordisms, etc.}, where each characteristic number $c^\#(E^{i+2}\to M^i)$ depends only on the homeomorphism type (resp. homotopy type, cobordism class, etc.) of the total space $E^{i+2}$, not on the choice of fibering of $E^{i+2}$.  As we will see below, being geometric with respect to these more general equivalences can be a strictly more difficult condition to satisfy than being geometric with respect to diffeomorphisms.

\bigskip
The first main theorem of this paper is the following.\pagebreak
\begin{theorem}[\textbf{Odd MMM classes are geometric}]
\label{theorem:MMMgeom}
For each $n\geq 1$ the Morita--Mumford--Miller class $e_{2n-1}$ is
geometric.
\end{theorem}

We prove Theorem~\ref{theorem:MMMgeom} in Section~\ref{section:geometric}.  Theorem~\ref{theorem:MMMgeom} leads to an obvious question:

\begin{question}
\label{question:even}
Are the even MMM classes $e_{2n}$ geometric?
\end{question}
In all examples that we have considered of surface bundles fibering in multiple ways, the even MMM classes $e_{2n}$ do not depend on the fibering, supporting an affirmative answer 
to Question~\ref{question:even}. In particular, this holds for the surface bundles constructed in Section~\ref{section:multifiberings}.   On the other hand, we will see later in the introduction that a slightly weaker version of Question~\ref{question:even} has a negative answer.

\para{Geometric with respect to cobordism} Theorem~\ref{theorem:MMMgeom} actually follows from the following more general result.    

\begin{theorem}[\textbf{Odd MMM classes are cobordism invariants}]
\label{theorem:novikov}
For any surface bundle $\Sigma\to E^{4n}\to M^{4n-2}$, the characteristic number $e_{2n-1}^\#(E\to M)$ can be expressed as the characteristic number corresponding to a certain explicit polynomial in the Pontryagin classes of $E$. It follows that for each $n\geq 1$ the MMM class $e_{2n-1}$ is geometric with respect to smooth cobordisms.
\end{theorem}
Pontryagin 
proved that the Pontryagin numbers are invariant under smooth cobordism, so if $\Sigma_g\to E^{4n}\to M^{4n-2}$ and $\Sigma_h\to \widetilde{E}^{4n}\to \widetilde{M}^{4n-2}$ are surface bundles and $E$ and $\widetilde{E}$ are smoothly cobordant, then the corresponding characteristic numbers coincide: \[e^\#_{2n-1}(E\to M)=e^\#_{2n-1}(\widetilde{E}\to \widetilde{M}).\]
Moreover Novikov \cite[Theorem 1]{No} proved that the rational Pontryagin classes are invariant under homeomorphism of smooth manifolds, so it suffices to assume that $\widetilde{E}$ is \emph{homeomorphic} to a manifold smoothly cobordant to $E$. (Note that this is weaker than saying that $\widetilde{E}$ is topologically cobordant to $E$.)
The polynomials involved in Theorem~\ref{theorem:novikov}, known as the \emph{Newton polynomials}, are rather complicated. For example, for $e_{11}$ and $\Sigma\to E^{24}\to M^{22}$ we will prove that (suppressing the ${\ }^\#$ superscripts on the right side for readability):
\[
\begin{array}{lll}
e^\#_{11}(E\to M)& = &p_1(E)^6-6p_1(E)^4p_2(E)+6p_1(E)^3p_3(E)+9p_1(E)^2p_2(E)^2\\
& & {}-6p_1(E)^2p_4(E)+6p_1(E)p_5(E)-12p_1(E)p_2(E)p_3(E)\\
& & {}-2p_2(E)^3+6p_2(E)p_4(E)+3p_3(E)^2-6p_6(E)
\end{array}\]
It is also possible to give a direct proof of Theorem~\ref{theorem:novikov} using the properties of the Chern character (see, e.g., \cite{MS}). However, we feel there is value in giving a self-contained proof of the theorem.

In contrast to Theorem~\ref{theorem:novikov}, we will see below that the {\em even} classes $e_{2n}$ are {\em not} geometric with respect to cobordism; indeed no polynomial in the $\{e_{2i}\}$ is geometric with respect to cobordism.  In Section~\ref{section:multifiberings} we give a general method of constructing closed oriented manifolds which fiber as surface bundles in more than one way.  This gives 
many explicit examples to which Theorem~\ref{theorem:novikov} applies.

\para{Other characteristic classes of surface bundles}  We can ask a much more general question than Question~\ref{question:even}.  Suppose $P(t_1,t_2,\ldots)$ is a homogeneous polynomial of total degree $j$, where $t_i$ has degree $i$.  For any surface bundle $\Sigma_g\to E^{2j+2}\to M^{2j}$ over a closed, oriented manifold $M^{2j}$, we have a well-defined characteristic number \[P^{\#}(E\to M)\coloneq \big\langle P(e_1(E\to M),e_2(E\to M),\ldots ), [M]\big\rangle\in \Z\] and so we can formulate the following question.

\begin{question}
\label{question:general}
Which homogeneous polynomials $P$ in the MMM classes $e_i$ are geometric?
\end{question}

Note that the property of being geometric is not \emph{a priori} preserved under cup product. Thus an affirmative answer to Question~\ref{question:even} combined with Theorem~\ref{theorem:MMMgeom} would still not imply that other polynomials in the MMM classes are geometric.  For example, we do not know whether or not $e_1^2$ is geometric.

\para{Handlebody bundles} The {\em genus $g$ handlebody} $V_g$ is the 3--manifold with $\partial V_g=\Sigma_g$ obtained as the regular neighborhood of a graph in $\R^3$.
Giansiracusa--Tillmann \cite[Theorem A]{GT} recently proved that the odd MMM classes $e_{2n-1}$ vanish for any surface bundle which bounds a handlebody bundle. As a corollary to  Theorem~\ref{theorem:novikov} we obtain another proof of this result in rational cohomology.
\begin{corollary}
\label{GTcorollary}
For any surface bundle $\Sigma_g\to E\to B$ over any base space $B$, if there exists a bundle $V_g\to W\to B$ whose fiberwise boundary is $\Sigma_g\to E\to B$, then $e_{2n-1}(E\to B)=0\in H^{4n-2}(B;\Q)$ for all $n\geq 1$.
\end{corollary}
\noindent The MMM classes can be defined in integral cohomology.  In this case the theorem of Giansiracusa--Tillmann still holds, but this is not implied by Theorem~\ref{theorem:novikov}.\pagebreak

As we remarked above, being geometric with respect to cobordism is a significantly more stringent condition than just being geometric. For example, we posed above the open question of whether the \emph{even} MMM classes are geometric. But as pointed out to us by Jeffrey Giansiracusa, handlebody bundles can be used to show that the even MMM classes are \emph{not} geometric with respect to cobordism, as follows. 

Hatcher \cite{Ha} recently announced an analogue of the theorem of Madsen--Weiss for handlebody bundles. His theorem implies that the stable characteristic classes of handlebody bundles (equivalently, of surface bundles which bound handlebody bundles) are exactly the polynomials in the even MMM classes, and that each such polynomial gives a nontrivial characteristic class. Appealing to a theorem of Thom (see the proof of Corollary~\ref{GTcorollary} for details),
 this implies that for every such polynomial $P(e_2,e_4,\ldots)$, there is a bundle $\Sigma_g\to E^{4n+2}\to B^{4n}$ which bounds $V_g\to W^{4n+3}\to B^{4n}$ but has $P(e_2(E\to B),e_4(E\to B),\ldots)\neq 0$. On the other hand, such a surface bundle is cobordant to the trivial bundle, and so its characteristic numbers must vanish.  We thus conclude the following.

\bigskip
{\it No polynomial in the even MMM classes is geometric with respect to cobordism.}
\bigskip

As the referee pointed out to us, if we relaxed the condition on surface bundles $\Sigma_g\to E\to B$ that the genus $g$ should be $\geq 2$, it could be seen more directly that the even MMM classes $e_{2n}$ are not geometric with respect to cobordism as follows. If $V\to B$ is an oriented 3-dimensional vector bundle with $p_1(V\to B)^n\neq 0\in H^{4n}(B)$, then the fiberwise sphere bundle $S^2\to SV\to B$ has $e_{2n}(SV\to B)\neq 0$. But this is the boundary of the disk bundle $D^2\to DV\to B$, so $SV$ is null-cobordant.

\para{Complex fibrations}
In contrast with this result, \emph{every} MMM class becomes geometric with respect to cobordism if we restrict our class of surface bundles to holomorphic fibrations. Let $C\to X^{n+1}\to Y^n$ be a holomorphic fibration, meaning that $X^{n+1}$ and $Y^n$ are closed complex manifolds of complex dimension $n+1$ and $n$ respectively, and the map $X^{n+1}\to Y^n$ is a holomorphic submersion. This is topologically a surface bundle, so we may consider the characteristic number $e_n^\#(X\to Y)\coloneq \langle e_n(X\to Y), [Y^n]\rangle \in \Z$. By expressing this characteristic number as a linear combination of Chern numbers, we prove the following theorem.

\begin{theorem}[\textbf{Every MMM class is geometric with respect to complex cobordism}]
\label{theorem:complex}
For each $n\geq 1$ the MMM class $e_n$ is geometric with respect to complex cobordism.  In other words, let $C\to X^{n+1}\to Y^n$ and $\widetilde{C}\to \widetilde{X}^{n+1}\to \widetilde{Y}^n$ be holomorphic fibrations.  If $X$ and $\widetilde{X}$ are cobordant as complex manifolds then the corresponding characteristic numbers coincide: \[e^\#_{n}(X\to Y)=e^\#_{n}(\widetilde{X}\to \widetilde{Y}).\]
\end{theorem}

\para{MMM classes as obstructions to fibering} Any characteristic class which is geometric gives an obstruction to the existence of a fibering with fiber a surface of small genus. For example,  we will use Theorem~\ref{theorem:MMMgeom} to prove the following corollary.

\begin{corollary}
\label{corollary:obstruction}
Let $n\geq 1$.  Let $E^{4n}\to M^{4n-2}$ be any surface bundle for which $e^{\#}_{2n-1}(E\to M)\neq 0$.  Then any fibering \[\Sigma_g\to E^{4n}\to N^{4n-2}\] of $E$ as a surface bundle must have $g >2n$.  In fact this holds for any fibering of any manifold topologically cobordant to $E^{4n}$.
\end{corollary}

\para{Outline of paper} In Section~\ref{section:geometric} we prove Theorem~\ref{theorem:MMMgeom}, Theorem~\ref{theorem:novikov}, and Theorem~\ref{theorem:complex} and deduce various corollaries.  
In Section~\ref{section:handlebody} we prove Corollary~\ref{GTcorollary}.  In Section~\ref{section:bordism} 
we give another perspective on Theorem~\ref{theorem:novikov} in terms of bordism groups.  In Section~\ref{section:multifiberings} we construct examples of closed manifolds fibering as surface bundles in more than one way.  
In Section~\ref{section:generalizations} we describe different ways to generalize the notion of geometric characteristic class, including vector bundles and bundles whose fibers are higher-dimensional manifolds.

\para{Acknowledgements}
We are grateful to Grigori Avramidi, Igor Belegradek, S\o ren Galatius, Ezra Getzler, Jeffrey Giansiracusa, Gabriele Mondello, Oscar Randal-Williams, and Shmuel Weinberger for helpful conversations. We  thank the referee for their careful reading of the paper and helpful suggestions.

\section{Geometric characteristic classes of surface bundles}
\label{section:geometric}
In the rest of the paper, all cohomology groups are with coefficients in $\Q$ unless otherwise specified, and all manifolds are smooth unless otherwise specified.

\subsection{Odd MMM classes are geometric}
In this section we prove Theorem~\ref{theorem:MMMgeom}, which we now recall.
\begin{theorem:main}
For each $n\geq 1$ the MMM class $e_{2n-1}$ is geometric.
\end{theorem:main}
We deduce Theorem~\ref{theorem:MMMgeom} from Theorem~\ref{theorem:novikov}, which we also recall.
\begin{theorem:cobordism}
For any surface bundle $\Sigma\to E^{4n}\to M^{4n-2}$, the characteristic number $e_{2n-1}^\#(E\to M)$ can be expressed as the characteristic number corresponding to a certain explicit polynomial in the Pontryagin classes of $E$. It follows that for each $n\geq 1$ the MMM class $e_{2n-1}$ is geometric with respect to smooth cobordisms.
\end{theorem:cobordism}
\begin{proof}[Proof of Theorems~\ref{theorem:MMMgeom} and \ref{theorem:novikov}]
Fix $n\geq 1$ for the rest of the proof, and consider a surface bundle $\Sigma\to
E^{4n}\overset{\pi}{\longrightarrow} M^{4n-2}$ whose base and total space are closed, smooth, oriented  manifolds. We will show that the characteristic number $e^\#_{2n-1}(E^{4n}\to M^{4n-2})$ can be written in terms of the Pontryagin numbers of $E=E^{4n}$, whose definition we recall below.

Recall that for a closed smooth manifold $E=E^{4n}$, the Pontryagin numbers of $M$  are the characteristic numbers of the tangent bundle $TE$ of $E$, defined as follows.  Let $p_i=p_i(TE)$ denote the $i$th Pontryagin class of (the tangent bundle of) $E$.  Let $J$ be a sequence $(j_1,\ldots,j_n)$ with $j_i\geq 0$ and $\sum_i i\cdot j_i=n$. Then the characteristic class $p_J\coloneq p_1^{j_1}p_2^{j_2}\cdots p_n^{j_n}$ yields, for any real vector bundle $V\to B$, a cohomology class $p_J(V\to B)\in H^{4n}(B)$. The \emph{Pontryagin number} $p^\#_J(E)$ associated to $J$ is the integer defined by \[p^\#_J(E)\coloneq \langle p_J(TE\to E),\ [E]\rangle\in \Z.\]

Let $\Sigma\to E\to M$ with $M$ be a surface bundle over a closed $2n$-manifold $M$.  
We deduce Theorem~\ref{theorem:MMMgeom} from  Theorem~\ref{theorem:novikov}, which states that we can write $e^{2n}\in H^{4n}(E)$ as a certain explicit polynomial in the Pontryagin classes $p_i(E)\coloneq p_i(TE)$.  As explained in the introduction, it follows immediately that  
$e^\#_{2n-1}$ only depends on the smooth cobordism type of $E$. This shows that $e_{2n-1}$ is geometric with respect to cobordism, and \emph{a fortiori} that $e_{2n-1}$ is geometric.

We now return to the MMM class $e_{2n-1}$. By definition \[e_{2n-1}^{\#}(E\to M)=\langle e_{2n-1}(E\to M),\ [M]\rangle = \int_M e_{2n-1}(E\to M).\] The definition of $e_{2n-1}$ and properties of the Gysin map give \[\int_M e_{2n-1}(E\to M)=\int_M\int_\Sigma e^{2n}=\int_E e^{2n}\] where $e$ is the Euler class of the $2$--plane bundle $T\pi$ given by the kernel of the differential $d\pi\colon TE\to TM$.   Note that $TE$ splits as a direct sum \[TE\ =\ T\pi\ \oplus\ \pi^*TM,\] so the total Ponytragin class $p=1+p_1+p_2+\cdots$ satisfies:
\begin{equation}
\label{eq:pofsplit}
\begin{array}{ll}
p(TE)\!\!&=\ p(T\pi)\cdot p(\pi^*(TM))\\
&=\ \big(1+p_1(T\pi)\big)\cdot\big(\pi^* p(TM)\big)
\end{array}
\end{equation}

Since $e\in H^2(E)$ is the Euler class of the 2--dimensional real vector bundle $T\pi\to E$, we have $e^2=p_1(T\pi)$. Collecting terms in \eqref{eq:pofsplit} thus gives 
\begin{equation}
\label{eq:piTE}
p_i(TE)=e^2\pi^*p_{i-1}(TM)+\pi^*p_i(TM).
\end{equation}
We will deduce the theorem from the following proposition.

\begin{proposition}
\label{prop:polys}
Let $\Z[x_1,x_2,\ldots]$ denote the ring of polynomials in variables $x_i$, graded so that $x_i$ has degree $i$. For each $n\geq 1$ there exists a unique 
polynomial $f_n\in \Z[x_1,\ldots,x_n]$, homogeneous of degree $n$, with the property that
\begin{equation}
\label{eq:fnproperty}
f_n(1+x_1,x_1+x_2,\ldots,x_{n-1}+x_n)=1+f_n(x_1,\ldots,x_n).
\end{equation}
\end{proposition}
We prove Proposition~\ref{prop:polys} in \S\ref{section:polys} below by explicitly constructing the polynomials $f_n$, which are known as the \emph{Newton polynomials}.  Assuming Proposition~\ref{prop:polys} for the moment, we complete the proof of the theorem.
Note that \[g_n(x_1,\ldots,x_n)\coloneq f_n(1+x_1,x_1+x_2,\ldots,x_{n-1}+x_n)\] is not homogeneous; the homogenization $\widetilde{g}_n$ of the polynomial $g_n$ is given by 
\[\widetilde{g}_n(t,x_1,\ldots,x_n)\coloneq f_n(t+x_1,tx_1+x_2,\ldots,tx_{n-1}+x_n),\] 
where $t$ has degree 1. Since $g_n(x_1,\ldots,x_n)=1+f_n(x_1,\ldots,x_n)$ it follows that 
$\widetilde{g}_n$ coincides with the homogenization of $1+f_n(x_1,\ldots,x_n)$. This gives the following identity:
\begin{equation}
\label{eq:fnt}
f_n(t+x_1,tx_1+x_2,\ldots,tx_{n-1}+x_n)=t^n+f_n(x_1,\ldots,x_n)
\end{equation}

For the sake of readability let $b_i\in H^{4i}(E)$ denote $\pi^*p_i(TM)$.  With this notation equation \eqref{eq:piTE} above becomes \[p_i(TE)=e^2b_{i-1}+b_i.\]
Now consider the cohomology class in $H^{4n}(E)$ defined by
\[f_n\big(p_1(TE),p_2(TE),\ldots,p_n(TE)\big)\in H^{4n}(E).\]
We then have
\begin{align*}
f_n\big(p_1(TE),p_2(TE),\ldots,p_n(TE)\big)
&=f_n\big(e^2+b_1,e^2b_1+b_2,\ldots,e^2b_{n-1}+b_n\big)\\
&=\widetilde{g}_n\big(e^2,b_1,b_2,\ldots,b_n\big)\\
&=e^{2n}+f_n\big(b_1,b_2,\ldots,b_n\big)\mathclap{\qquad\qquad\qquad\qquad\qquad\qquad\qquad\qquad\qquad\text{by \eqref{eq:fnt}}}\\
&=e^{2n}+f_n\big(\pi^*p_1(TM),\pi^*p_2(TM),\ldots,\pi^*p_n(TM)\big)\\
&=e^{2n}+\pi^* f_n\big(p_1(TM),\ldots,p_n(TM)\big)
\end{align*}
But $f_n$ has degree $n$ and $p_i(TM)$ lies in $H^i(M)$, so the term $f_n(p_1(TM),\ldots,p_n(TM))$ lies in $H^{4n}(M)$. Since $M$ is a $(4n-2)$--dimensional manifold this term vanishes. We conclude that
\begin{equation}
\label{eq:fne2n}
f_n\big(p_1(TE),p_2(TE),\ldots,p_n(TE)\big)=e^{2n}
\end{equation}
so that
\begin{equation}
\label{eq:charnum2}
e^\#_{2n-1}(E\to M)=\int_E e^{2n}=\int_E f_n\big(p_1(TE),p_2(TE),\ldots,p_n(TE)\big).
\end{equation}

The righthand side of \eqref{eq:charnum2} is a linear combination of Pontryagin numbers. Indeed, write \[f_n(x_1\ldots,x_n)=\sum_{J} a_Jx^J\] with $a_J\in \Z$, where $x^J$ denotes the multinomial $x_1^{j_1}\cdots x_n^{j_n}$. We then have
\[e^\#_{2n-1}(E\to M)=\sum_J a_J\cdot p_J^\#(E).\]
Pontryagin proved that these characteristic numbers 
$p_J^\#(E)$ only depend on the smooth cobordism type of $E$. This shows that $e_{2n-1}$ is geometric with respect to cobordism, and \emph{a fortiori} that $e_{2n-1}$ is geometric.
\end{proof}

\subsection{Constructing the polynomials $f_n$}
\label{section:polys}
\begin{proof}[Proof of Proposition~\ref{prop:polys}]
After giving a recursive definition of the Newton polynomials $f_n$  we verify that they have the desired property \eqref{eq:fnproperty}: 
\[{f_n(1+x_1,\ldots,x_{n-1}+x_n)}=1+f_n(x_1,\ldots,x_n).\] 
We conclude by showing that any $f_n$ satisfying this property is unique.

%The first remark to make is that the property \eqref{eq:fnproperty} is equivalent to the \textit{a priori} weaker property that $f_n(1+x_1,\ldots,x_{n-1}+x_n)$ is congruent to 1 modulo terms of degree $n$, which is all we needed in the proof of Theorem~\ref{theorem:MMMgeom}. To see this, note that every term in $f_n(1+x_1,\ldots,x_{n-1}+x_n)$ has degree at most $n$, and expanding this polynomial by choosing one term from each input $x_{k-1}+x_k$, the terms of degree $n$ are exactly those for which the $x_k$ term is chosen every time (all other terms in the expansion will have lower degree); but this is exactly $f_n(x_1,\ldots,x_n)$. We will need the stronger property \eqref{eq:fnproperty} for the inductive step.

Recursively define $f_n\in \Z[x_1,\ldots,x_n]$ by $f_1\coloneq x_1$ and 

\begin{equation}
\label{eq:fndef}
f_{n}\coloneq \sum_{k=1}^{n-1}(-1)^{k-1}x_kf_{n-k}\ +\ (-1)^{n-1}nx_{n}
\end{equation}
for $n> 1$.

The first few polynomials $f_n$ are:
\[
\begin{array}{l}
f_1=x_1\\
f_2=x_1^2-2x_2\\
f_3=x_1^3-3x_1x_2+3x_3\\
f_4=x_1^4-4x_1^2x_2 +4x_1x_3+ 2x_2^2-4x_4\\
f_5=x_1^5-5x_1^3x_2+5x_1^2x_3+5x_1x_2^2-5x_1x_4-5x_2x_3+5x_5\\
f_6=x_1^6-6x_1^4x_2+6x_1^3x_3+9x_1^2x_2^2-6x_1^2x_4+6x_1x_5\\
\qquad\qquad\qquad\quad\;{}-12x_1x_2x_3-2x_2^3+6x_2x_4+3x_3^2-6x_6
%f_7=x_1^7-7x_1^5x_2+7x_1^4x_3+14x_1^3x_2^2-7x_1^3x_4+7x_1^2x_5\\
%\qquad\qquad\qquad\quad\;{}-21x_1^2x_2x_3-7x_1x_2^3+14x_1x_2x_4+7x_1x_3^2\\
%\qquad\qquad\qquad\quad\;{}-7x_1x_6+7x_2^2x_3-7x_2x_5-7x_3x_4+7x_7
\end{array}
\]

We prove by induction that the polynomials $f_n$ have the property \[{f_n(1+x_1,\ldots,x_{n-1}+x_n)}=1+f_n(x_1,\ldots,x_n).\] The claim is trivial for $n=1$ 
since \[f_1(1+x_1)=1+x_1=1+f_1(x_1).\] The inductive step will only use the recursive definition \eqref{eq:fndef} of $f_n$. As in the proof of Theorem~\ref{theorem:MMMgeom} we define $g_n\in \Z[x_1,\ldots,x_n]$ by \[g_n\coloneq f_n(1+x_1,\ldots,x_{n-1}+x_n).\]
Assume that $g_k=1+f_k$ for $k<n$; our goal is to prove that $g_n=1+f_n$.  For simplicity we set $x_0=1$. Then we compute:
\begin{align*}
g_n&=f_n(1+x_1,\ldots,x_{n-1}+x_n)\\
&=\sum_{k=1}^{n-1}(-1)^{k-1}(x_{k-1}+x_k)(1+f_{n-k})+(-1)^{n-1}n(x_{n-1}+x_n)\\
&=\sum_{k=1}^{n-1}(-1)^{k-1}(x_{k-1}+x_k)\quad+(-1)^{n-1}x_{n-1}\\
&\quad+\sum_{k=1}^{n-1}(-1)^{k-1}x_{k-1}f_{n-k}\quad+(-1)^{n-1}(n-1)x_{n-1}\\
&\quad+\sum_{k=1}^{n-1}(-1)^{k-1}x_{k}f_{n-k}\qquad+(-1)^{n-1}nx_n\\
&=x_0+\left(f_{n-1}-\sum_{j=1}^{n-2}(-1)^{j-1}x_jf_{n-1-j}-(-1)^{n-2}(n-1)x_{n-1}\right)+f_n\\
&=1+f_{n-1}-f_{n-1}+f_n=1+f_n.
\end{align*}
This completes the proof that the $f_n$ constructed above satisfy property \eqref{eq:fnproperty}.

Since we will not need the uniqueness of the $f_n$ in this paper, we only sketch its proof. If $f'_n$ were another solution to \eqref{eq:fnproperty} for some $n$, the difference $h_n\coloneq f_n-f'_n$ would be a homogeneous polynomial of degree $n$ satisfying \[h_n(1+x_1,\ldots,x_{n-1}+x_n)=(1+f_n)-(1+f'_n)=h_n.\] But $h_n(1+x_1,\ldots,x_{n-1}+x_n)$ cannot be homogenous of degree $n$, as we now show. Let the \emph{pseudo-degree} of a polynomial in $\Z[x_1,\ldots,x_n]$ be its degree when the variables $x_i$ are given degree 1, instead of degree $i$ as usual. Let $k$ be the pseudo-degree of $h_n$.  We necessarily have $1\leq k\leq n$. Then the smallest degree terms in $h_n(1+x_1,\ldots,x_{n-1}+x_n)$ are of degree $n-k<n$. Indeed each leading term of $h_n$ (those of pseudo-degree $k$) contributes a linearly independent term of degree $n-k$. This contradiction shows that $h_n=0$, and so $f_n$ is unique.
\end{proof}

%We remark that the coefficients of $f_n$ have the following closed form, which is not hard to prove by induction. For the monomial $x_1^{j_1}x_2^{j_2}\cdots x_n^{j_n}$ (we set $j=j_1+\cdots+j_n$), the coefficient is $(-1)^{j}$ times the weighted number of sequences $(a_1,\ldots,a_j)$ with $j_i$ entries equal to $i$, with the sequence $(a_1,\ldots,a_j)$ weighted by the integer $a_1$. For example, for $x_1^2x_2^2$ the sequences are \[(1,1,2,2),\ (1,2,1,2),\ (1,2,2,1),\ (2,1,1,2),\ (2,1,2,1),\ (2,2,1,1)\] so the coefficient of $x_1^2x_2^2$ in $f_6$ is $1+1+1+2+2+2=9$.

The Newton polynomials are traditionally defined implicitly by the relation
\[f_n(e_1,\ldots,e_n)=x_1^n+\cdots+x_n^n,\] where $e_i\in \Z[x_1,\ldots,x_n]$ is the $i$-th elementary symmetric function. However, it seems to be more difficult to prove the key property \eqref{eq:fnproperty} from this definition.

\subsection{All MMM classes are geometric for complex fibrations}
\begin{proof}[Proof of Theorem~\ref{theorem:complex}]
Given a complex fibration $C\to X^{n+1}\to Y_n$, we will prove the theorem by showing that $e_n^\#(X\to Y)$ can be expressed as a combination of the Chern numbers of $X$. Since the Chern numbers are invariant under complex cobordism, this will complete the proof.

Our argument will be very similar to the proof of Theorem~\ref{theorem:novikov}. Note that $TX$ splits as a direct sum \[TX\ =\ T\pi\ \oplus\ \pi^*TY,\] and $c_1(T\pi)=e\in H^2(X)$, so just as in \eqref{eq:piTE} we have
% the total Chern class $c=1+c_1+c_2+\cdots$ satisfies:
%\begin{equation}
%%\label{eq:pofsplit}
%\begin{array}{ll}
%c(TX)\!\!&=\ c(T\pi)\cdot c(\pi^*(TY))\\
%&=\ \big(1+c_1(T\pi)\big)\cdot\big(\pi^* c(TY)\big)
%\end{array}
%\end{equation}
%
%Since $e\in H^2(E)$ is the Euler class of the 2--dimensional real vector bundle $T\pi\to E$, we have $e^2=p_1(T\pi)$. Collecting terms in \eqref{eq:pofsplit} thus gives 
%\begin{equation}
%\label{eq:piTE}
\[c_i(TX)=e\cdot \pi^*c_{i-1}(TY)+\pi^*c_i(TY).
\] Just as in the derivation preceding \eqref{eq:fne2n}, the property \eqref{eq:fnt} of the Newton polynomials implies that
\[\begin{split}f_{n+1}\big(c_1(TX),c_2(TX),\ldots,c_n(TX),c_{n+1}(TX)\big)\\=e^{n+1} + \pi^* f_{n+1}\big(c_1(TY),\ldots,c_n(TY),c_{n+1}(TY)\big)\end{split}
\] in $H^{2n+2}(X)$. Since $Y=Y^n$ has real dimension $2n$, the latter term vanishes, and we conclude that $e^{n+1}$ can be expressed as a polynomial $f_{n+1}$ in the Chern classes of $X$. Since $e_n^\#(X\to Y)=\langle e^{n+1},[X]\rangle$, this shows that $e_n^\#(X\to Y)$ can be written as a fixed linear combination of the Chern numbers of $X$, as desired.
\end{proof}

\subsection{Odd MMM classes vanish on handlebody bundles}
\label{section:handlebody}

In this section we deduce Corollary~\ref{GTcorollary}, originally proved by Giansiracusa--Tillmann \cite[Theorem A]{GT}, from Theorem~\ref{theorem:novikov}.
\begin{corollary:GT}If there exists a bundle $V_g\to W\to B$ whose fiberwise boundary is $\Sigma_g\to E\to B$, then $e_{2n-1}(E\to B)=0\in H^{4n-2}(B;\Q)$ for all $n\geq 1$.
\end{corollary:GT}
\begin{proof}
When the base is a closed manifold $B^{4n-2}$ of dimension $4n-2$, it follows immediately from Theorem~\ref{theorem:novikov} that $e_{2n-1}(E\to B)=0$. Indeed, Theorem~\ref{theorem:novikov} implies that $e_{2n-1}(E\to B)=0$ whenever $E^{4n}=\partial W^{4n+1}$, whether or not $W$ is a handlebody bundle, or a fiber bundle at all. When the base space $B$ is only a CW complex, consider an arbitrary homology class $x\in H_{4n-2}(B;\Z)$; we will show that $\langle e_{2n-1}(E\to B), x\rangle =0$.

Thom \cite[Theorem II.29]{Th} proved  that every homology class in a closed orientable manifold has an integral multiple which can be represented by (the fundamental class of) a closed submanifold. This can be strengthened to show that every homology class in any CW complex $X$ has an odd integral multiple which can be represented by the fundamental class $f_*[M]$ of a closed manifold $M$ under a continuous map $f\colon M\to X$,
see e.g. Conner \cite[Corollary 15.3]{Co}.

It follows that for some nonzero $k\in \Z$, the homology class $k\cdot x$ is represented by a map $M^{4n-2}\to B$.
Let $\Sigma_g\to E'\to M$ be the pullback of $\Sigma_g\to E\to B$. By naturality:
\begin{equation}\label{eq:Thom}\begin{array}{r}
\langle e_{2n-1}(E'\to M),[M]\rangle = \langle e_{2n-1}(E\to B),[M]\rangle\\
=\langle e_{2n-1}(E\to B),k\cdot x\rangle\\=k\cdot \langle e_{2n-1}(E\to B),x\rangle\end{array}\end{equation}
However, the pullback $V_g\to W'\to M$ of $V_g\to W\to B$ has fiberwise boundary $\Sigma_g\to E'\to M$. Thus by the argument at the beginning of the proof, Theorem~\ref{theorem:novikov} implies that $\langle e_{2n-1}(E'\to M),[M]\rangle=0$. By \eqref{eq:Thom}, this implies that $\langle e_{2n-1}(E\to B),x\rangle=0\in \Q$.
\end{proof}
As we mentioned in the introduction, Giansiracusa--Tillmann in fact proved the vanishing of $e_{2n-1}$ in \emph{integral} cohomology, which is not implied by Theorem~\ref{theorem:novikov}.

\subsection{Geometric classes as obstructions}
As noted in the introduction, any characteristic class which is geometric gives an obstruction to the existence of a fibering for which the fiber is a surface of small genus. For example, we have the following corollary to the main theorems.
\begin{corollary:obstruct}
For $n\geq 1$, let $E^{4n}\to M^{4n-2}$ be a surface bundle with
$e^{\#}_{2n-1}(E\to M)\neq 0$.  Then any fibering \[\Sigma_g\to E^{4n}\to N^{4n-2}\] of $E$ as a surface bundle must have $g > 2n$.  In fact this holds for any fibering of any manifold topologically cobordant to $E^{4n}$.
\end{corollary:obstruct}
\begin{proof}
The space $\BDiff^+(\Sigma_g)$ mentioned in the introduction is the classifying space for bundles with structure group $\Diff^+(\Sigma_g)$. From this it follows that characteristic classes of $\Sigma_g$--bundles correspond to cohomology classes in $H^*(\BDiff^+(\Sigma_g))$.

For $g\geq 2$, Looijenga \cite{Lo} proved that any polynomial of total degree $k\geq g-1$ in the MMM classes $e_i$ vanishes in $H^{k}(\BDiff^+(\Sigma_g);\Q)$. In particular, $e_{2n-1}\in H^{4n-2}(\BDiff^+(\Sigma_g);\Q)$ vanishes for $g\leq 2n$.  Any fibering $\Sigma_g\to E\to N$ with $g\leq n$ would have $e^\#_{2n-1}(E\to N)=0$, but by Theorem~\ref{theorem:MMMgeom} this contradicts 
  our assumption that $e^\#_{2n-1}(E\to M)\neq 0$. Appealing to Theorem~\ref{theorem:novikov}, we obtain the same bound $g>2n$ for 
  any fibering $\Sigma_g\to \widetilde{E}\to N'$ with $\widetilde{E}$ topologically cobordant to $E$.
\end{proof}

It is clear that for complex fibrations Theorem~\ref{theorem:complex} implies a similar result.

\subsection{Bordism groups}
\label{section:bordism}
In this section we give a different perspective on Theorem~\ref{theorem:novikov} by bundling manifolds and surfaces together into cobordism classes. The oriented bordism group $\Omega^{\SO}_n$ is the abelian group of cobordism classes of oriented $n$--manifolds. We can similarly consider the abelian group of surface bundles over oriented $n$--manifolds modulo cobordism of bundles. This group is denoted $\Omega^{\SO}_n(\coprod\BDiff^+(\Sigma_g))$, as we now briefly explain.

Oriented bordism gives an extraordinary homology theory $\Omega^{\SO}_*$, where $\Omega^{\SO}_n(X)$ is the abelian group of cobordism classes of oriented $n$--manifolds $M$ equipped with a map to $X$. The space $\BDiff^+(\Sigma_g)$ is the classifying space for bundles with structure group $\Diff^+(\Sigma_g)$, which means that homotopy classes of maps from $B$ to $\BDiff^+(\Sigma_g)$ correspond to isomorphism classes of $\Sigma_g$--bundles over $B$. This property implies that $\Omega^{\SO}_n(\BDiff^+(\Sigma_g))$ is the group of $\Sigma_g$--bundles with base an oriented $n$--manifold, modulo cobordism as $\Sigma_g$--bundles. Taking all $g$ together gives $\Omega^{\SO}_n(\coprod\BDiff^+(\Sigma_g))$, the oriented bordism group of surface bundles.\\

For each $n\geq 0$ there is a natural map \[\Omega^{\SO}_n(\coprod\BDiff^+(\Sigma_g))\to \Omega^{\SO}_{n+2}\]  given by forgetting the bundle structure:  a 
surface bundle $[\Sigma_g\to E^{n+2}\to B^n]$ is sent to the cobordism class of the total space $[E^{n+2}]$. This is one version of the Gysin map. Any $i$--dimensional characteristic class $c$ of surface bundles yields a map $$c^\#\colon \Omega^{\SO}_i(\coprod\BDiff^+(\Sigma_g))\to \Z$$ by taking the corresponding characteristic number. We can then rephrase the definition of ``geometric with respect to cobordism'' as follows.
\begin{definition}[Definition~\ref{def} for cobordism, restated]
\label{def:cob}
A characteristic class $c$ of surface bundles is \emph{geometric with respect to cobordism} if and only if the map $c^\#\colon \Omega^{\SO}_i(\coprod\BDiff^+(\Sigma_g))\to \Z$ factors through the map $\Omega^{\SO}_i(\BDiff^+(\Sigma_g))\to \Omega^{\SO}_{i+2}$:
\[\xymatrix{
   \Omega^{\SO}_i(\coprod\BDiff^+(\Sigma_g))\ar[r]\ar_{c^\#}[dr]& \Omega^{\SO}_{i+2}\ar@{-->}[d]\\
  &\Z}\]
\end{definition}
For example, Theorem~\ref{theorem:novikov} states that  the map $e_{2n-1}^\#\colon \Omega^{\SO}_{4n-2}(\coprod\BDiff^+(\Sigma_g))\to \Z$ induced by the MMM class $e_{2n-1}$ factors through $\Omega^{\SO}_{4n}$. Conversely, the result of Hatcher \cite{Ha}  described in the introduction implies  that the map  $e_{2n}^\#\colon \Omega^{\SO}_{4n}(\coprod\BDiff^+(\Sigma_g))\to \Z$ does not factor through $\Omega^{\SO}_{4n+2}$. (For another proof, recall that Thom proved that  $\Omega^{\SO}_{4n+2}$ consists entirely of torsion, while $e_{2n}^\#$ is well-known to be rationally nontrivial.)

We emphasize that Definition~\ref{def:cob} is strictly stronger than the notion of geometric characteristic class itself (Definition~\ref{def}), which cannot be interpreted in this way. There is no reason that a characteristic number which only depends on the diffeomorphism class of the total space should necessarily  depend only on the \emph{cobordism class} of the total space.

\section{Examples of bundles fibering in multiple ways}
\label{section:multifiberings}
In this section we explain a general construction of closed, oriented manifolds that fiber as a surface bundle in more than one way.  This construction thus provides many families of examples to which %Theorem~\ref{theorem:MMMgeom} applies.  
Theorem~\ref{theorem:novikov} applies.

Given a surface bundle $\Sigma\to N^n\to B^{n-2}$ with section $s\colon B\to N$,\footnote{The condition that $\Sigma\to N\to B$ has a section is not very restrictive, since any surface bundle $\Sigma\to B\to Y$ induces a surface bundle $\Sigma\to B\times_Y B\to B$ with section $s\colon B\to B\times_Y B$ given by the diagonal embedding.} we will construct a manifold $E^{n+2}$ with distinct fiberings \[\Sigma_g\to E^{n+2}\to P^{n}\quad\text{and}\quad\Sigma_h\to E^{n+2}\to Q^{n}.\] In fact, by varying the parameters in the construction, each surface bundle $\Sigma\to N^n\to B^{n-2}$ yields an infinite family of such manifolds $E^{n+2}$.

One caveat is that the total space $E$ that we produce is only canonically defined as a topological manifold. We will find a smooth structure on $E$ so that $\Sigma_g\to E\to P$ is a smooth fibering, and a smooth structure on $E$ so that $\Sigma_h\to E\to Q$ is a smooth fibering, so that  Theorem~\ref{theorem:novikov} applies to these fiberings. However, the argument does not imply that the resulting smooth structures on $E$ are diffeomorphic (although we believe that they are).

\para{The construction of the Atiyah--Kodaira manifold} To explain the various steps in the construction, we first sketch the construction of the Atiyah--Kodaira manifold $M_\Sigma$ from a surface $\Sigma$. Fix an integer 
$k\geq 2$.

Choose a finite-sheeted normal covering $S\to \Sigma$, and let $f_1,\ldots,f_k\colon S\to S$ be distinct deck transformations. Let $\widehat{S}\overset{\pi}{\longrightarrow} S$ be the $k^{2g(S)}$--sheeted normal cover determined by 
\[1\to \pi_1(\widehat{S})\to \pi_1(S)\to H_1(S;\Z/k\Z)\to 1.\] Consider the 4--manifold $\widehat{S}\times S$, and define the codimension--2 submanifolds $\Delta_1,\ldots,\Delta_k$ to be the graphs of $f_i\circ \pi\colon \widehat{S}\to S$: \[\Delta_i=\big\{(p,q)\in \widehat{S}\times S\big|f_i(\pi(p))=q\big\}\] Since the maps $f_i$ are pointwise distinct, these graphs are disjoint.
Using the K\"{u}nneth formula, one can check that the cohomology class represented by $[\Delta_1]+\cdots+[\Delta_k]\in H^2(\widehat{S}\times S;\Z)$ is divisible by $k$. It follows (see \cite[Proposition 6]{Hi}) that $\widehat{S}\times S$ admits a cyclic $k$--sheeted cover branched over the union $\Delta\coloneq\Delta_1\cup\cdots \cup \Delta_k$. (As we will see below, this description does not uniquely determine the branched cover, but all the different covers yield the same total space.) This branched cover $M_\Sigma\to \widehat{S}\times S$ is the desired 4--manifold $M_\Sigma$.

Note that the composition $M_\Sigma\to \widehat{S}\times S\to \widehat{S}$ is a surface bundle, whose fiber is the cyclic $k$--sheeted cover $\Sigma_h$ of $S$ branched over $\Delta\cap S$, which consists of $k$ points. Similarly, the composition $M_\Sigma\to \widehat{S}\times S\to S$ is a surface bundle, whose fiber is the cyclic $k$--sheeted cover $\Sigma_g$ of $\widehat{S}$ branched over $\Delta\cap \widehat{S}$, which consists of $k^{2g(S)+1}$ points. 

\para{Atiyah-Kodaira for families} We now describe a procedure which begins with a surface bundle $\Sigma\to N\to B$ with section $s\colon B\to N$, and constructs a bundle $M_\Sigma\to E\to B'$, where $B'$ is a finite cover of $B$. One artifact of this construction is that although $M_\Sigma$ is a smooth manifold, the resulting $M_\Sigma$--bundles have structure group $\Homeo^+(M_\Sigma)$ and thus are not necessarily smooth.

We begin by taking $\Gamma=\Diff^+(\Sigma,\ast)$, which is the structure group for $\Sigma$--bundles endowed with a section. At various stages in the proof we will need to replace $\Gamma$ by a finite index subgroup. In the end we will have a homomorphism $\phi\colon\Gamma\to \Homeo^+(M_\Sigma)$ for some finite index subgroup $\Gamma<\Diff^+(\Sigma,\ast)$. For any  surface bundle $\Sigma\to N\to B$, by pulling back to a finite cover of the base $B'\to B$ we can reduce the structure group to $\Gamma$. We then apply the usual fiber-replacement determined by the homomorphism of structure groups $\phi\colon\Gamma\to \Homeo^+(M_\Sigma)$ to the $\Sigma$--bundle $\Sigma\to N'\to B'$ to obtain an $M_\Sigma$--bundle $M_\Sigma\to E\to B'$. We will then use the two fiberings $M_\Sigma\to \widehat{S}$ and $M_\Sigma\to S$ to construct two fiberings of $E$.

\para{Defining the homomorphism \boldmath$\Gamma\to \Homeo^+(M_\Sigma)$}
We maintain the same choices as in the previous section, so $S\to \Sigma$ is a finite-sheeted normal cover and $f_1,\ldots,f_k\colon S\to S$ are distinct  deck transformations. Let $G$ be the deck group of the cover $S\to \Sigma$; we additionally assume that each $f_i$ is nontrivial. We fix a lift in $S$ of $\ast\in \Sigma$, which by abuse of notation we also denote $\ast\in S$. A diffeomorphism $\varphi\in\Diff(\Sigma,\ast)$ lifts to $\Diff(S,\ast)$ if and only if it preserves the finite index subgroup $\pi_1(S,\ast)\lhd \pi_1(\Sigma,\ast)$. For such $\varphi$, we have an induced action on the quotient $\pi_1(\Sigma,\ast)/\pi_1(S,\ast)\approx G$. We replace $\Gamma$ by the finite-index subgroup of $\Diff^+(\Sigma,\ast)$ consisting of diffeomorphisms $\varphi$ which preserve $\pi_1(S,\ast)$ and act trivially on the quotient $G$.

Our choice of a basepoint $\ast \in S$ gives an unambiguous lifting $\Gamma\hookrightarrow \Diff(S,\ast)$. As before $\pi\colon \widehat{S}\to S$ is the characteristic cover with deck group $H_1(S;\Z/k\Z)$. Since this cover is characteristic, every diffeomorphism of $S$ lifts to $\widehat{S}$. Choosing a basepoint $\ast\in \widehat{S}$, we obtain a lifting $\Diff(S,\ast)\hookrightarrow \Diff(\widehat{S},\ast)$. 
Since $\Gamma$ acts trivially on $G$, the image of $\Gamma$ in $\Diff(S,\ast)$ commutes with the action of $G$ by deck transformations; in particular, the image of $\Gamma$ commutes with $f_i$. Thus the diagonal action of $\Gamma\subset \Diff(S,\ast)$ on $\widehat{S}\times S$ preserves the subsets $\Delta_i=\big\{\big(p,f_i\circ\pi(p)\big)\big\}$ and  their union $\Delta$.

At this point the action of $\Gamma$ on $\widehat{S}\times S$ fixes the basepoint $(\ast,\ast)\in \widehat{S}\times S$ (which does not lie on $\Delta_i$ since $f_i$ is nontrivial) and preserves $\Delta$. Let $\gamma_i$ be the small circle 
\[\gamma_i\coloneq\big\{\big(\theta,f_i\circ\pi (\ast)\big)\big|\dist(\theta,\ast)=\epsilon\}\] linking once with $\Delta_i$. A cyclic $k$--sheeted cover of $\widehat{S}\times S$ branched over $\Delta$ is determined by a homomorphism \[\pi_1(\widehat{S}\times S \setminus \Delta)\to \Z/k\Z\qquad\text{satisfying}\quad \gamma_i\mapsto 1\in \Z/k\Z\ \text{ for all }i.\] As our final modification, we replace $\Gamma$ by the finite-index subgroup which acts trivially on ${H^1(\widehat{S}\times S\setminus \Delta;\Z/k\Z)}$. (If $k=2$, for convenience later we consider instead the action on $H^1(\widehat{S}\times S\setminus \Delta;\Z/4\Z)$.) Fix a homomorphism $\pi_1(\widehat{S}\times S \setminus \Delta)\to \Z/k\Z$. This determines a branched cover $M_\Sigma\to \widehat{S}\times S$. Fix a basepoint $\ast\in M_\Sigma$ lying above $(\ast,\ast)$. Our modification guarantees that $\Gamma$ preserves the kernel of this homomorphism, so choosing the lift fixing $\ast\in M_\Sigma$, the action $\Gamma\hookrightarrow\Diff(\widehat{S}\times S\setminus \Delta)$ lifts to $\Diff(M_\Sigma\setminus \Delta)$.

By continuity we can extend the resulting diffeomorphisms across $\Delta$ to all of $M_\Sigma$ (the action of $\Gamma$ on each component $\Delta_i$ of the ramification locus can be identified with the action of $\Gamma$ on $\widehat{S}$). It is easy to check that the resulting map of $M_\Sigma$ is orientation-preserving.  However, the extended map will not be smooth along the ramification locus $\Delta\subset M_\Sigma$. Thus we have defined the desired homomorphism
\[\phi\colon\Gamma\to \Homeo^+(M_\Sigma).\]
We finish by emphasizing a key property: since $\Gamma$ acts on $\widehat{S}\times S$ diagonally, the projection $\widehat{S}\times S\to S$ is $\Gamma$--equivariant, as is the projection $\widehat{S}\times S\to \widehat{S}$. Since the action on $M_\Sigma$ lifts the action on $\widehat{S}\times S$, the fiberings $M_\Sigma\to S$ and $M_\Sigma\to \widehat{S}$ are also $\Gamma$--equivariant.

\para{Fibering as surface bundles}
As mentioned before, for any surface bundle $\Sigma\to N\to B$ we can reduce the structure group to $\Gamma$ by passing to a finite cover $B'\to B$. Then applying $\phi\colon \Gamma\to \Homeo^+(M_\Sigma)$, we obtain a bundle $M_\Sigma\to E\to B'$.

Let $P$ be the total space of the bundle $S\to P\to B'$, and $Q$ the total space of the bundle $\widehat{S}\to Q\to B'$, determined by the maps $\Gamma\hookrightarrow\Diff(S)$ and $\Gamma\hookrightarrow\Diff(\widehat{S})$ respectively.
Since the fibering $M_\Sigma\to S$ is $\Gamma$--equivariant, it induces a fibering $E\to P$; the fiber $\Sigma_g$ of $\Sigma_g\to M_\Sigma\to S$ is also the fiber of $\Sigma_g\to E\to P$. Similarly, the $\Gamma$--equivariant fibering $\Sigma_h\to M_\Sigma\to \widehat{S}$  induces a fibering $\Sigma_h\to E\to Q$.

\para{Smoothing the resulting bundles}
Since the structure group of $M_\Sigma\to E\to B'$ is $\Homeo^+(M_\Sigma)$, we only know that $E$ is a topological manifold. But to apply the results of this paper, we need $E\to P$ and $E\to Q$ to be smooth fiberings of smooth manifolds. Fortunately this is possible, as we illustrate for $\Sigma_g\to E\to P$.\\

We know that the map $E\to P$ is a topological fiber bundle, with fiber $\Sigma_g$ and structure group $\Homeo^+(\Sigma_g)$. But for $g\geq 2$ the identity component $\Homeo_0(\Sigma_g)$ is contractible \cite[Theorem 1.14]{FM}, so the structure group can be taken to be the discrete group $\Mod(\Sigma_g)\coloneq \pi_0\Homeo^+(\Sigma_g)$. (We remark that $\Diff_0(\Sigma_g)$ is also contractible, so $\Mod(\Sigma_g)\approx\pi_0\Diff^+(\Sigma_g)$ as well.)
Since $\Mod(\Sigma_g)$ is discrete, it follows that a surface bundle $\Sigma_g\to E\to P$ is determined by its \emph{monodromy representation} $\rho\colon \pi_1(P)\to \Mod(\Sigma_g)$.

We will need the fact that for our construction, the image of $\rho$ is torsion-free. This can be seen as follows.
Let $R=P\times_{B'}Q$ be the total space of the bundle $\widehat{S}\times S\to R\to B'$, so that we have a bundle $\widehat{S}\to R\to P$. From the bundle $M_\Sigma\to E\to B'$ we see that $E$ is a branched cover of $R$; writing $\Sigma_g\to E\to P$, we see that $E$ is obtained from $\widehat{S}\to R\to P$ by a fiberwise branched cover. It follows that the monodromy $\rho_E\colon \pi_1(P)\to \Mod(\Sigma_g)$ factors through the monodromy $\rho_R\colon \pi_1(P)\to \Mod(\widehat{S})$.

Recall that $P$ is the total space of the bundle $S\to P\to B'$. The restriction of the bundle $\widehat{S}\to R\to P$ to $S\subset P$ is just the product $\widehat{S}\to \widehat{S}\times S\to S$. Thus $\rho_R$ vanishes on the fundamental group $\pi_1(S)$ of the fiber, so we may think of $\rho_R$ as a map $\pi_1(B')\to \Mod(\widehat{S})$, induced by the action $\Gamma\to \Diff(\widehat{S})$ described earlier. Our assumptions on $\Gamma$ imply that the image of the monodromy $\pi_1(B')\to\Mod(\widehat{S})$ is contained in the subgroup $\Mod(\widehat{S})[k]$ which acts trivially on $H^1(\widehat{S};\Z/k\Z)$. Indeed we assumed that 
that $\Gamma$ acts trivially on $H^1(\widehat{S}\times S\setminus\Delta;\Z/k\Z)$. Since $H^1(\widehat{S};\Z/k\Z)$ injects into $H^1(\widehat{S}\times S\setminus\Delta;\Z/k\Z)$, this implies the claim. For $k=2$ we made a further assumption which implies the image is contained in $\Mod(\widehat{S})[4]$.

The subgroup $\Mod(\widehat{S})[k]$ is known to be torsion-free for $k\geq 3$ \cite[Theorem 6.9]{FM}. Note that $P$ is a finite cover of $N$, so it is naturally a smooth manifold. Now the desired property follows from the following proposition.

\begin{proposition}
\label{prop:smoothing}
For any surface bundle $\Sigma\to E\to B$, if $B$ is a smooth manifold and the image of the monodromy representation $\pi_1(B)\to \Mod(\Sigma)$ is torsion-free, there is a smooth structure on $E$ making $\Sigma\to E\to B$ into a smooth bundle of smooth manifolds.
\end{proposition}
\noindent Note that the hypothesis of the proposition is always satisfied after passing to some finite cover $B'\to B$ (corresponding for $\Mod(\Sigma)[k]$ for some $k\geq 3$, for example).

\begin{proof}[Proof of Proposition~\ref{prop:smoothing}]
If $\Gamma$ is a torsion-free subgroup of $\Mod(\Sigma_g)$, a finite-dimensional model for the classifying space $B\Gamma$ is given by the appropriate cover $\mathcal{M}^\Gamma_g$ of $\mathcal{M}_g$, the \emph{moduli space} of genus $g$ Riemann surfaces \cite[\S12.6]{FM}. The classifying space $\mathcal{M}^\Gamma_g$ is a smooth manifold of dimension $6g-6$, and there is a universal surface bundle $\mathcal{C}^\Gamma_g$ fitting into a smooth bundle of smooth manifolds $\Sigma_g\to \mathcal{C}^\Gamma_g\to \mathcal{M}^\Gamma_g$. Since $\mathcal{M}^\Gamma_g$ is the classifying space $B\Gamma$, we know that $\Sigma_g$--bundles over a base $X$ with structure group $\Gamma$ correspond to maps $f\colon X\to \mathcal{M}^\Gamma_g$ up to homotopy. This correspondence is made explicit by sending the map $f$ to the pullback bundle $\Sigma_g\to f^*\mathcal{C}^\Gamma_g\to X$; homotopic maps yield isomorpic bundles.

The bundle $\Sigma_g\to E\to B$ of the proposition is classified by some map $f\colon B\to \mathcal{M}^\Gamma_g$. Since every continuous map between smooth manifolds can be approximated by a smooth map, let $\widetilde{f}\colon B\to \mathcal{M}^\Gamma_g$ be a smooth map homotopic to $f$. The pullback of a smooth bundle along a smooth map is smooth, making the pullback $\Sigma_g\to \widetilde{f}^*\mathcal{C}^\Gamma_g\to B$ into a smooth bundle. Since this is isomorphic to the bundle $\Sigma_g\to E\to B$, we can think of this as giving a smooth structure on $E$ with respect to which $\Sigma_g\to E\to B$ is a smooth bundle of smooth manifolds, as desired.
\end{proof}

\section{Variations on geometric classes}
\label{section:generalizations}
\subsection{Geometric classes for bundles of the ``wrong'' dimension}

Our definition of geometric for a $d$--dimensional characteristic class $c$ 
required us to look at a closed, oriented $d$--manifold as base space, since we needed a canonical homology class on which to evaluate $c$, thus obtaining a number which could compare for various fibering.  However, when the base manifold has dimension different from $d$, or is no longer closed or orientable, there is still a reasonable notion of ``geometric''.

\begin{definition}[\textbf{Weakly geometric characteristic classes}]
A $d$--dimensional characteristic class $c$ for surface bundles is \emph{weakly geometric} if the vanishing or  nonvanishing of $c$ does not depend 
on the fibering.  More precisely, whenever 
   \[\Sigma_g\to E\to B\] and \[\Sigma_h \to E'\to B'\]
   are two surface bundles with homeomorphic total spaces $E\approx E'$, then 
   $c(E\to B)=0$ if and only if $c(E'\to B')=0$. 
\end{definition}
Note that we do \emph{not} assume that $B$ is a closed $d$--manifold, only that it is a manifold. Thus despite the name, geometric does not imply weakly geometric. For example, we do not know if the odd MMM classes $e_{2n-1}$ are weakly geometric.
\begin{question}
\label{question:weakgeom}
Which MMM classes are weakly geometric for surface bundles with base a closed manifold of any dimension?
\end{question}

 For the first MMM class $e_1$, Question~\ref{question:weakgeom} is equivalent to the following question. If $E\to M$ is a surface bundle, the preimage of a closed surface in $M$ is a 4--manifold in $E$. Let $E\to M$ be a surface bundle so that some closed surface in $M$ has preimage with nonzero signature. If $E\to M'$ is another fibering of $E$ as a surface bundle, must there exist such a surface in $M'$?  
There is also the following \emph{a priori} stronger question:  if $N$ is a $4$--dimensional 
submanifold of $E$ and if $E\to M$ restricts to $N$ as a submersion, must 
every fibering $E\to M'$ restrict to $N$ as a submersion? An affirmative answer would imply that $e_1$ is weakly geometric, by showing that the nontriviality of $e_1$ is carried on submanifolds $N$ which cannot be avoided by choosing another fibering.

\subsection{Generalized MMM classes}
There are higher-dimensional versions of the Morita--Mumford--Miller classes, which are characteristic classes of orientable bundles with fiber a $d$--dimensional manifold. 
For any bundle
$M^{d}\to E\to B$ we have the $d$--dimensional vertical tangent bundle $T_ME$, and for any polynomial $P$ in the Euler class and Pontryagin classes of $T_ME$, integrating along the fiber yields a characteristic class $\varepsilon_P$. We define such a characteristic class to be geometric if the associated characteristic number depends only on the total space, exactly as in Definition~\ref{def}.

Ebert \cite[Theorem B]{Eb} has proved that if $d$ is even, every polynomial in the characteristic classes $\varepsilon_P$ is nonzero for some bundle $M^d\to E^{n+d}\to B^n$, and for $d$ odd he described exactly which polynomials in the $\varepsilon_P$ vanish.  In light of the results of this paper, the following question seems quite natural.

\begin{question}
Which polynomials in the higher MMM classes $\varepsilon_P$ are geometric? In particular, is $\varepsilon_{e^2}\in H^d$ geometric?
\end{question}

One possible conjecture is that  $\varepsilon_P$ should be geometric with respect to cobordism whenever $P$ is a polynomial only in the Pontryagin classes. This is supported by Giansiracusa--Tillman, who prove for such $P$ in \cite[Corollary C]{GT} that $\varepsilon_P$ is invariant with respect to \emph{fiberwise} cobordism.

\subsection{Vector bundles}
It is natural to ask whether the notion of ``geometric characteristic classes'' can be extended to vector bundles. We first point out that there are many examples of distinct vector bundles with diffeomorphic total space, so this question is far from vacuous.

Mazur \cite{Ma} proved for compact manifolds that if there is a tangential homotopy equivalence $M^k\to N^k$, then $M\times \R^n$ and $N\times \R^n$ are diffeomorphic for sufficiently large $n$. In particular, although the 3--dimensional lens spaces $L(7,1)$ and $L(7,2)$ are not homeomorphic, the trivial bundle $L(7,1)\times \R^4$ is diffeomorphic to $L(7,2)\times \R^4$. In a similar vein, Siebenmann proved \cite[Theorem 2.2]{Si} that if $\R^n\to E\to M$ is a vector bundle of rank $n>2$ over a compact manifold $M$, and $N$ is any smooth compact submanifold of $E$ so that $N\hookrightarrow E$ is a homotopy equivalence, then there is a vector bundle $\R^n\to V\to N$ so that $V$ is diffeomorphic to $E$.

The soul theorem of Cheeger--Gromoll \cite{CG} states that associated to any complete non-negatively curved metric on a Riemannian manifold $E$ is a compact totally geodesic submanifold $M\subset E$ (called the \emph{soul} of $E$) so that $E$ is diffeomorphic to the total space of the normal bundle $NM$. By varying the metric on $E$, this provides numerous examples of distinct vector bundles with diffeomorphic total space; for example, Belegradek \cite[Theorem 1]{Be} shows there exist infinitely many compact 7--manifolds $M_i$ with vector bundles $\R^5\to E_i\to M_i$ whose base spaces $M_i$ are pairwise non-homeomorphic, but whose total spaces $E_i$ are all diffeomorphic to $S^3\times S^4\times \R^5$.\\

It is not difficult to show that the Euler class is geometric in the sense of this paper: if $\R^n\to E\to M^n$ and $\R^n\to E\to N^n$ are vector bundles with the same total space, the Euler numbers $e^\#(E\to M)=\langle e(E\to M),[M]\rangle$ and $e^\#(E\to N)=\langle e(E\to N),[N]\rangle$ coincide. However for vector bundles we can strengthen the notion of geometric characteristic class by dropping the restriction that our bundles have base space an $n$--dimensional manifold. For surface bundles, our focus on characteristic numbers was necessary because there is no way to directly compare the cohomology groups of different base spaces. But for any vector bundles  $\R^n\to E\to M$ and $\R^n\to E\to N$, the canonical homotopy equivalences $E\to M$ and $E\to N$ induce an identification $H^*(M)\approx H^*(N)$ by which we can compare characteristic classes directly.

For the Euler class, it is proved in Belegradek--Kwasik--Schultz \cite[Proposition 5.1]{BKS} that $e(E\to M)\in H^n(M)$ and $e(E\to N)\in H^n(N)$ coincide under this identification. It would be very interesting to know which other characteristic classes of vector bundles are geometric in this stronger sense; for the rational Pontryagin classes this is established in \cite[Proposition 5.3]{BKS} for bundles $\R^n\to E\to M$ when $n\leq 3$.

\small

\noindent
Department of Mathematics\\Stanford University\\
450 Serra Mall\\
Stanford, CA 94305\\
E-mail: church@math.stanford.edu\\

\noindent
Department of Mathematics\\ University of Chicago\\
5734 University Ave.\\
Chicago, IL 60637\\
E-mail:  farb@math.uchicago.edu, matt\_tbo@math.uchicago.edu

\end{document}